\documentclass[11pt, reqno]{amsart}
\usepackage{amssymb,pb-diagram}
\usepackage{amsrefs}
\usepackage{fancyhdr}
\usepackage{color}
\usepackage{amsmath}

\newtheorem{theorem}{Theorem}[section]
\newtheorem{proposition}[theorem]{Proposition}
\newtheorem{corollary}[theorem]{Corollary}
\newtheorem{lemma}[theorem]{Lemma}

\theoremstyle{definition}

\theoremstyle{remark}
\newtheorem{remark}[theorem]{Remark}

\newtheorem{question}[theorem]{Question}

\theoremstyle{definition}

\newcommand{\ds}{\displaystyle}

\newcommand{\R}{\ensuremath{\mathbb{R}}}
\newcommand{\N}{\ensuremath{\mathbb{N}}}

\newcommand{\Z}{\ensuremath{\mathbb{Z}}}

\newcommand{\mellp}{\ensuremath{(\ell_p,d_p)}}
\newcommand{\mclp}{\ensuremath{(L_p,d_p)}}
\newcommand{\mellq}{\ensuremath{(\ell_q,d_q)}}
\newcommand{\mclq}{\ensuremath{(L_q,d_q)}}
\newcommand{\redp}{\ensuremath{(\mathbb{R},d_p)}}
\newcommand{\redq}{\ensuremath{(\mathbb{R},d_q)}}
\newcommand{\sd}{\ensuremath{\mathsf{S}_{\mathsf{D}}}}

\newcommand{\nsd}{\ensuremath{\mathsf{NS}_{\mathsf{D}}}}
\newcommand{\nst}{\ensuremath{\mathsf{NS}_{\mathsf{T}}}}
\newcommand{\notisom}{\ensuremath{\underset{\approx}{\lhook\joinrel\relbar\joinrel\not\rightarrow}}}

\newcommand{\coarse}{\ensuremath{\underset{coarse}{\lhook\joinrel\relbar\joinrel\rightarrow}}}
\newcommand{\uniform}{\ensuremath{\underset{unif}{\lhook\joinrel\relbar\joinrel\rightarrow}}}

\newcommand{\isom}{\ensuremath{\underset{\approx}{\lhook\joinrel\relbar\joinrel\rightarrow}}}
\newcommand{\isometric}{\ensuremath{\underset{=}{\lhook\joinrel\relbar\joinrel\rightarrow}}}
\newcommand{\lisometric}{\ensuremath{\underset{\cong}{\lhook\joinrel\relbar\joinrel\rightarrow}}}
\newcommand{\lip}{\ensuremath{\underset{Lip}{\lhook\joinrel\relbar\joinrel\rightarrow}}}
\newcommand{\notlip}{\ensuremath{\underset{Lip}{\lhook\joinrel\relbar\not\joinrel\relbar\joinrel\rightarrow}}}

\newcommand{\uh}{\ensuremath{\underset{unif}{\sim}}}

\newcommand{\ce}{\ensuremath{\underset{coarse}{\sim}}}

\newcommand{\lipe}{\ensuremath{\underset{Lip}{\sim}}}
\newcommand{\almost}{\ensuremath{\underset{1+\epsilon}{\lhook\joinrel\relbar\joinrel\rightarrow}}}
\newcommand{\notalmost}{\ensuremath{\underset{1+\epsilon}{\lhook\joinrel\relbar\joinrel\not\relbar\joinrel\rightarrow}}}

\def\supp{\text{supp}}

\allowdisplaybreaks

\begin{document}

\title[Embedding snowflaked metrics and applications]{Embeddability of snowflaked metrics  with applications to the nonlinear geometry of the spaces $\mathbf{L_p}$ and $\mathbf{\ell_{p}}$  for $\mathbf{0<p<\infty}$}

\thanks{This  article  grew out of the authors stay  at the Mathematical Sciences Research Institute in Berkeley, California,  as  research members in the {\it Quantitative Geometry\/} program held in the Fall of 2011, partially supported by  an NSF grant DMS-0932078 administered by the MSRI. The authors also acknowledge the support of the Spanish Ministerio de Ciencia e Innovaci\'on Research Grant {\it Operadores, ret\'{\i}culos, y geometr\'{\i}a de espacios de Banach}, reference number MTM2008-02652/MTM}

\author[F. Albiac]{F. Albiac}\address{Mathematics Department\\ Universidad P\'ublica de Navarra\\
 Pamplona\\ 31006 Spain} \email{fernando.albiac@unavarra.es}

\author[F. Baudier]{F. Baudier}
\address{Mathematics Department\\ Texas A\&M University\\College Station, TX 77843 USA}
\email{florent@math.tamu.edu}

\subjclass[2000]{46B80, 46A16, 46T99}

\keywords{Lipschitz, uniform, and coarse embedding, snowflaked metric, Enflo-type, Hausdorff dimension}

\begin{abstract} 
We study the classical spaces $L_{p}$ and $\ell_{p}$ for the whole range $0<p<\infty$ from a metric viewpoint. As we go along,  we look over some of the results and techniques that, together with our work in this paper, have permitted to obtain a complete Lipschitz embeddability roadmap between any two of those spaces when equipped with both their ad-hoc distances and their snowflakings. Through connections with weaker forms of embeddings that lead to basic (yet fundamental) open problems, we also set the challenging goal of understanding the dissimilarities between the well-known subspace structure and the different nonlinear geometries that coexist inside $L_{p}$ and $\ell_{p}$.

\end{abstract}

\maketitle

\section{Introduction}
\subsection{Motivation}\

\medskip
\noindent If one is given a metric space it is natural to look at the metric geometry of the space at different scales. For instance, if in a non-discrete metric space we care only about small distances between points and properties that are stable under uniform embeddings we will refer to this particular geometry of  the space as {\it uniform geometry}. Similarly, the {\it coarse geometry} of an unbounded metric space will deal with large distances and properties stable under coarse embeddings. Finally the {\it Lipschitz geometry} accounts for the metric geometry at all scales and the behavior of Lipschitz embeddings. The Lipschitz geometry of finite metric spaces is a central theme in theoretical computer science and the design of algorithms and has been thoroughly investigated since the 90's both from functional analysts and computer scientists. The uniform and Lipschitz geometry of Banach spaces has been studied in some form or another since the rise of the 20th Century and culminated in the publication of the authoritative book of Benyamini and Lindenstrauss \cite{BenyaminiLindenstrauss2000} in  2000.  Introduced by Gromov in \cite{Gromov1993}, the coarse geometry of finitely generated groups turns out to be a key concept in noncommutative geometry in connection to the Baum-Connes and Novikov Conjectures as demonstrated by Yu \cite{Yu2000}. It is of great interest to understand what kind of metric spaces (mostly finitely generated groups) can be coarsely embeddable into some ``nice'' Banach spaces (e.g. Hilbert spaces). Despite the good deal of work that has been done in understanding the coarse geometry of the domain spaces, namely groups, we have a quite narrow picture regarding the coarse geometry of the target spaces, even if they are taken amongst classical Banach spaces. The main motivation of the authors is to initiate a systematic study of the coarse geometry of {\it general metric spaces}. It is clearly an ambitious task and in this paper we will attack this project under a specific angle which seems to us to be a natural place to start.

\subsection{Notation and terminology}\

\medskip
\noindent Let $(\mathcal M_{1},d_{1})$ and $(\mathcal M_{2},d_{2})$ be two unbounded metric spaces. In  nonlinear theory we are interested in knowing whether or not  we can find a copy of  $\mathcal M_{1}$  inside $\mathcal M_{2}$ in  one of the following senses:

\medskip

$\bullet$  $\mathcal M_{1}$ {\it Lipschitz embeds}  into $\mathcal M_{2}$, and we write it $\mathcal M_{1}\lip \mathcal M_{2}$ for short, if there exists a one-to-one map  $f:\mathcal M_{1}\to \mathcal M_{2}$, which for some constants $A, B>0$ satisfies  
\begin{equation}\label{Lipembedding}
A^{-1} d_{1}(x,y)\le d_{2}(f(x),f(y))\le B d_{1}(x,y), \qquad \forall x,y\in \mathcal M_{1}. 
\end{equation}
Equivalently, a 1-1 map $f:\mathcal M_{1}\to \mathcal M_{2}$ is a Lipschitz embedding if  $\omega_f(t)\le B t$ and $\rho_f(t)\ge {t}/{A}$  for some constants $A, B>0$ and all $t>0$, where
$$\rho_f(t)=\inf\{d_{2}(f(x),f(y)) : d_{1}(x,y)\geq t\},$$
and 
$$\omega_f(t)={\rm sup}\{d_{2}(f(x),f(y)) : d_{1}(x,y)\leq t\}.$$

\medskip

 $\bullet$ $\mathcal M_{1}$ {\it uniformly  embeds}  into $\mathcal M_{2}$, which is denoted by $\mathcal M_{1}\uniform \mathcal M_{2}$, if there  is an injective, uniformly continuous  map  $f:\mathcal M_{1}\to \mathcal M_{2}$ whose inverse $f^{-1}:f(\mathcal M_{1})\subset \mathcal M_{2}\to \mathcal M_{1}$ is also uniformly continuous. This equates to asking  the injective map $f$ that $\rho_f(t)>0$ for all
$t>0$ and $\lim_{t\to 0}\omega_f(t)=0$.  A uniform embedding imposes a uniformity on how the map and its inverse change distances locally. Sometimes, e.g., when the domain space and its image under $f$ are metrically convex, this implies uniformity on the changes that the embedding makes to distant points.  However, although a Banach space is  metrically convex, its image under a uniform embedding may not be so. 

\medskip

 $\bullet$ $\mathcal M_{1}$ {\it coarsely embeds}  into $\mathcal M_{2}$, denoted by $\mathcal M_{1}\coarse \mathcal M_{2}$,  if there is $f:\mathcal M_{1}\to \mathcal M_{2}$ so that  $\omega_f(t)<\infty$ for all
$t>0$ and $\displaystyle \lim_{t\to \infty}\rho_f(t)=\infty$. It is perhaps worth mentioning that a coarse embedding need not be injective nor continuous, hence this kind of embedding overlooks the structure of a metric space in the neighborhood of a point. Indeed, in contrast to uniform embeddings,  coarse embeddings only capture the structure of the space at large scales. 
\medskip

Aside from these, we will write $\mathcal M_{1}\isometric  \mathcal M_{2}$ if there exists an {\it isometric embedding} from $\mathcal M_{1}$ into $\mathcal M_{2}$, and  use $\mathcal M_{1}\almost \mathcal M_{2}$ to denote  {\it almost isometric embeddability}  of $\mathcal M_{1}$ into $\mathcal M_{2}$, i.e., if for every $\epsilon>0$ there exists a Lipschitz embedding $f_{\epsilon}:\mathcal M_{1}\to \mathcal M_{2}$ so that the distortion of the embedding, measured by the product  of the optimal constants $A$, $B$ in \eqref{Lipembedding} is smaller than $1+\epsilon$.

Two metric spaces $\mathcal M_{1}$, $\mathcal M_{2}$ will be referred to as {\it Lipschitz isomorphic} (also, Lipschitz  equivalent), {\it uniformly homeomorphic}, or  {\it coarsely homeomorphic}, if we can find a bijective embedding $f:\mathcal M_{1}\to \mathcal M_{2}$ of the specific kind. For short we will write $\mathcal M_{1} \lipe \mathcal M_{2}$, $\mathcal M_{1} \uh \mathcal M_{2}$, and $\mathcal M_{1} \ce \mathcal M_{2}$, respectively.

\medskip

Our main subject of study in this article will be the nonlinear embeddings that occur between any two members of the classical sequence spaces
$\ell_p$
and the function spaces $L_p=L_{p}[0,1]$,  for the whole range of $0<p<\infty$, when equipped either with their
 usual  distances
or their snowflakings.
Recall that for $0<p\le 1$,   the standard distances in $\ell_{p}$ and $L_{p}$ are respectively given by 
 $$d_{\ell_p}(x,y)=\sum_{n=1}^\infty \vert x_n-y_n\vert^p,$$
and $$d_{L_p}(f,g)= \int_0^1 \vert f(t)-g(t)\vert^p dt,$$
 whereas for  values of  $p$ on the other side of the spectrum, $p\ge 1$, the distance in the spaces  is induced by their norms $\Vert x\Vert_{\ell_p}=(\sum_{n=1}^\infty \vert x_n\vert^p)^{1/p},$
and $\Vert f\Vert_{L_p} = (\int_0^1 \vert f(t)\vert^p dt)^{1/p}.$
Note that for $p=1$ we have $(L_1,\Vert\cdot\Vert_{L_1})=(L_1,d_{L_1})$ and $(\ell_1,\Vert\cdot\Vert_{\ell_1})=(\ell_1,d_{\ell_1})$, hence for those two metric spaces we will drop the distances  and just write without confusion $L_1$ or $\ell_1$. For simplicity,  we will  unify the notation for the four metrics introduced above, e.g., when we write $(L_p,d_p)$ for $0<p<\infty$ it shall be understood that we endow $L_p$ with the metric $d_{L_p}(f,g)$ if $0<p\le 1$ or with the metric $\Vert f-g\Vert_{L_p}$ if $p\ge 1$.

\subsection{Organization of the paper}\

\medskip

We have divided this article in five more sections, each of which is rendered as self-contained as possible.  The flow of the exposition has a deliberate survey flavor  that, we hope, will smooth the way for understanding the topics we cover and will help the reader put the new results in the right place.

\smallskip

When a Lipschitz embedding between two metric spaces is ruled out, it is only natural  to determine whether  there exist weaker embeddings, e.g. coarse, uniform, or quasisymmetric embeddings to name a few. For $0<s<1$, the $\mathbf s$-{\it snowflaked} version of a metric space $(\mathcal M,d)$ is the metric space $(\mathcal M,d^s)$, sometimes denoted $\mathcal M^{(s)}$. It is clear that $(\mathcal M,d)$ and $(\mathcal M,d^s)$ are coarsely and uniformly equivalent and it is easy to show that they are quasisymmetrically equivalent. However, they are rarely Lipschitz equivalent. Another important remark is that a Lipschitz embedding of some snowflaking of a metric space induces an embedding of the original metric space which is simultaneously a coarse, uniform, and quasisymmetric embedding. In Section~\ref{Section2} we introduce three new classes of metric spaces in order to study the following general question:

\begin{question}\label{snow}
Let $0<s_1\neq s_2\le1$.
Under what conditions is it possible  to Lipschitz embed $(\mathcal M,d^{s_1})$
into $(\mathcal M,d^{s_2})$?
\end{question}

 To tackle this problem we will analyze how some metric invariants such as the Hausdorff dimension, Enflo type, or roundness may thwart the embeddability of snowflakings of general metric spaces.

\smallskip

 In the third section, the theme is embedding snowflakings of the real line into the metric spaces $\ell_{p}$ for $0<p\le 1$. Although our approach is infinite-dimensional in nature, it is inspired by the work of Assouad \cite{Assouad1983} on Lipschitz embeddability of $(\mathbb R, |\cdot|^{p})$ for $0<p\le 1$ into the  finite-dimensional space $\ell_{q}^{N}$ for $q\ge 1$ equipped with the standard distance.
 
\smallskip

In Section~\ref{Section5} we complete the picture of the  Lipschitz embedding theory between the  spaces $L_{p}$ and $\ell_{p}$ for $0<p<\infty$ and connect it with the unique Lipschitz subspace structure problems.

\smallskip

In the fifth  section we discuss the embeddability of snowflakings of the spaces $L_p$ and $\ell_{p}$ for $0<p<\infty$, and  show that the Mendel and Naor's isometric embeddings between $L_{p}$ spaces obtained in \cite{MendelNaor2004} have an $\ell_{p}$-space counterpart in the Lipschitz category. In light of these results we are able to give partial answers to  Question~\ref{snow}.
 
\smallskip

We close with a brief section devoted to bridge our work with  a few known results on coarse and uniform embeddings between the metric spaces $L_{p}$ and $\ell_{p}$  for $0<p<\infty$, where we also state a few open problems that seem the natural road to take to further research in this direction.

\section{Embedding snowflakings of metric spaces}\label{Section2}

\noindent Let us introduce the following three new classes of metric spaces:
\medskip

$\bullet$ ${\mathsf S_{\mathsf D}}= \Big\{(\mathcal M,d) : (\mathcal M,d^s)\lip (\mathcal M,d)\;\text{for all}\, 0<s\le1\ \Big\}$, i.e., the collection of metric spaces $\mathcal M$ that contain a subset Lipschitz equivalent to $\mathcal M^{(s)}$ for all $0<s<1$. 

\smallskip

$\bullet$ $\nsd=\Big\{(\mathcal M,d) : (\mathcal M,d^s)\notlip (\mathcal M,d)\; \text{for any}\, 0<s<1\Big\}$, that is, the class of metric spaces that cannot be the target space of a Lipschitz embedding from any $\mathcal M^{(s)}$.

\smallskip

$\bullet$ $\nst=\Big\{(\mathcal M,d) : (\mathcal M,d)\notlip (\mathcal M,d^s)\; \textrm{for any}\; 0<s<1\Big\}$, formed by those metric spaces that are not Lipschitz equivalent to any subset of their snowflakings.

\smallskip

To get started, using the metric invariance of Enflo-type and of Hausdorff dimension, we will be able to exhibit a few first members of the classes $\sd,\nsd,\nst$ and give rather general restrictions regarding the existence of a Lipschitz embedding between any two metric spaces.

\smallskip

 First we  describe briefly the notion of Enflo type,  which was  introduced by Enflo in \cite{Enflo1978} although it seems as if the term was coined in  \cite{Pisier1986b}  by Pisier.
 An $n$-dimensional {\it cube} in an arbitrary metric space $\mathcal{M}$ is a collection of $2^{n}$ not necessarily distinct  points  $C=\{x_u\}_{u\in\{-1,1\}^n}$ in $\mathcal{M}$, where each point $x_u$ in $C$ is indexed by a distinct vector $u\in\{-1,1\}^n$.   If $C$ is an $n$-dimensional cube in $\mathcal{M}$, a {\it diagonal} in $C$ is  an unordered pair of the form $\{x_u,x_{-u}\}$, i.e., a pair of vertices in $C$ whose indexing vectors differ in all their coordinates. The set of all the diagonals in $C$ will be denoted by
 $D(C)$. An {\it edge} in $C$ is an unordered pair $\{x_u,x_v\}$ such that $u$ and $v$ differ in only one coordinate. The set of all edges of $C$ is denoted by $E(C)$. Then, a  metric space $(\mathcal M, d)$ is said to have {\it Enflo-type} ${\mathbf p}\ge 1$ if there exists a constant $K>0$ such that for every $n\in\N$ and for every $n$-dimensional cube $C\subset\mathcal{M}$ the sum of the lengths of the $2^{n-1}$ diagonals in $C$ is related to the sum of the lengths of the $n2^{n-1}$ edges in $C$ by the formula

\begin{equation}\label{etype}
\displaystyle\sum_{\{a,b\}\in D(C)}d(a,b)^p\le K^p\displaystyle\sum_{\{a,b\}\in E(C)}d(a,b)^p.\end{equation}

Every metric space has Enflo-type $1$ with constant $K=1$ by the triangle inequality, so we will put
$$\textrm{E-type}(\mathcal M)=\sup\{p : \mathcal M\;\; \textrm{has Enflo-type $p$}\}.$$ 
   A metric space $\mathcal{M}$ is said to have {\it finite (supremum) Enflo-type} if $\textrm{E-type}(\mathcal M)<\infty$.

\smallskip

The first assertion of our next Lemma makes Enflo-type a powerful tool to study Lipschitz embeddability between general metric spaces. In turn,  the second assertion  will be extremely relevant when dealing with snowflakings. 

\begin{lemma}\label{enflolemma}  Let $(\mathcal M_{1},d_{1})$, $(\mathcal M_{2},d_{2})$ be metric spaces.
\begin{enumerate}
\item[(i)] If $(\mathcal M_{1},d_{1})\lip(\mathcal M_{2},d_{2})$ then  $\textrm{E-type}(\mathcal M_{1})\ge \textrm{E-type}(\mathcal M_{2})$.
\item[(ii)] Let $0<s<1$.  Then $\ds \textrm{E-type}(\mathcal M_{1}^{(s)})=\frac{\textrm{E-type}(\mathcal M_{1})}{s}.$\end{enumerate}
\end{lemma}

\begin{proof} (i) Assume $(\mathcal M_{2},d_{2})$ satisfies \eqref{etype} for some $p\ge 1$. Let $f\colon \mathcal M_{1}\to \mathcal M_{2}$ and $A,B>0$ such that $${A^{-1}}d_{1}(x,y)\le d_{2}(f(x),f(y))\le Bd_{1}(x,y),\quad \forall x,y\in \mathcal M_{1}.$$ Let $n\in\N$ and $C$ be a $n$-dimensional cube in $\mathcal{M}_1$. Then $f(C)=\{f(x_u): u\in\{-1,1\}^n\}$ is a $n$-dimensional cube in $\mathcal{M}_2$ and 
\begin{align*}
\displaystyle\sum_{\{a,b\}\in D(f(C))}d(a,b)^p & \le K^p\displaystyle\sum_{\{a,b\}\in E(f(C))}d(a,b)^p \\
&\le K^p\displaystyle\sum_{\{a,b\}\in E(C)}d(f(a),f(b))^p\\
& \le K^pB^p\displaystyle\sum_{\{a,b\}\in E(C)}d(a,b)^p
\end{align*}
But,
\begin{align*}
\displaystyle\sum_{\{a,b\}\in D(f(C))}d(a,b)^p & = \displaystyle\sum_{\{a,b\}\in D(C)}d(f(a),f(b))^p \\
&\ge A^{-p}\displaystyle\sum_{\{a,b\}\in D(C)}d(a,b)^p\\
\end{align*}
and so
\begin{align*}
\displaystyle\sum_{\{a,b\}\in D(C)}d(a,b)^p & \le A^pK^pB^p \displaystyle\sum_{\{a,b\}\in D(C)}d(a,b)^p \\
\end{align*}
(ii) follows readily from the definition of Enflo-type.
\end{proof}

  In the next straightforward lemma we see that  the Hausdorff dimension verifies an analogue of Lemma~\ref{enflolemma}  (ii) and the reverse inequality with respect to Lipschitz embeddings. We will not attempt to include here the subtle (and long!) definition of Hausdorff dimension and instead we prefer to refer to \cite{BuragoBuragoIvanov2001} and \cite{Heinonen2001} for an extended discussion of this notion.

\begin{lemma} Let $(\mathcal M_1,d_1)$ and $(\mathcal M_2,d_2)$ be metric spaces, and $0<s<1$.
\begin{enumerate}
\item[(i)] If $(\mathcal M_1,d_1)\lip(\mathcal M_2,d_2)$, then $\textrm{dim}_{\mathcal{H}}(\mathcal M_1)\le \textrm{dim}_{\mathcal{H}}(\mathcal M_2)$.
\item[(ii)] $\ds \textrm{dim}_{\mathcal{H}}(\mathcal M_1^{(s)})=\ds\frac{\textrm{dim}_{\mathcal{H}}(\mathcal M_1)}{s}$.
\end{enumerate}
\end{lemma}

The last two lemmas put together lead to the following proposition.
\begin{proposition}\label{restriction}
Let $(\mathcal M_{1},d_{1})$, $(\mathcal M_{2},d_{2})$ be  metric spaces, and let $0<s_1,s_2\le1$. If $(\mathcal M_{1},d_{1}^{s_1})\lip (\mathcal M_{2},d_{2}^{s_2})$, then $$\frac{\textrm{E-type}(\mathcal M_{1})}{s_1}\ge \frac{\textrm{E-type}(\mathcal M_{2})}{s_2},$$ and $$\frac{{\text dim}_{\mathcal{H}}(\mathcal M_{1})}{s_1}\le \frac{{\text dim}_{\mathcal{H}}(\mathcal M_{2})}{s_2}.$$
\end{proposition}

We are now able to partially answer Question \ref{snow} when $0<s_2<1$ and $s_1=1$, and when $s_2=1$ and $0<s_1<1$.

\begin{corollary}\label{Enflo}
Let $0<s<1$ and $(M,d)$ be a metric space with finite Enflo-type, then
$(M,d)\notlip (M,d^s)$.
\end{corollary}

\begin{proof}
Put $p=\textrm{E-type}(\mathcal M)$. We have $1\le p<\infty$, hence if $(\mathcal M,d)\lip (\mathcal M,d^s)$ it follows from Proposition \ref{restriction} that $p/s\le p$, which contradicts  $0<s<1$.
\end{proof}

It follows from Corollary \ref{Enflo} that the class $\nst$ is a rather large class of metric spaces. Indeed, any metrically convex space (i.e., every pair of points has metric midpoints), or more generally any metric space containing a line segment has Enflo-type less than 2. In particular CAT(0)-spaces, for instance metric trees, are in $\nst$. In Section~\ref{Section4} (\textsection\ref{Section4.2}) we will show that all the $L_p$-spaces and $\ell_{p}$-spaces for $0<p<\infty$ belong to $\nst$. We can prove the following corollary along the same lines.

\begin{corollary}\label{Hausdorff}
Let $0<s<1$ and $(\mathcal M,d)$ be a metric space with finite positive Hausdorff dimension. Then
$(\mathcal M,d^s)\notlip (\mathcal M,d)$
\end{corollary}

\begin{remark} Note that any metric space with finite Hausdorff-dimension, in particular any finite dimensional Banach space since the $n$-dimensional Euclidean space has Hausdorff dimension $n$, is in the class $\nsd$.
\end{remark}

In the case of  almost isometric embeddings we have a more precise description of those metric spaces not embeddable into any of their snowflakings. For this purpose we use the concept of roundness introduced by  Enflo in \cite{Enflo1970}, where it is applied to prove that an $L_{p}(\mu)$-space is not uniformly equivalent to $L_{q}(\nu)$ if $1\le p \not=q\le 2$.  A metric space $(\mathcal M, d)$ is said to have {\it roundness} $ p\ge 1$ if the following generalization of the triangle law of the distance is fulfilled  for every four points $a_{1},a_{2},a_{3},a_{4}\in \mathcal M$,
\begin{equation}\label{roundness}
d(a_{1},a_{3})^p+d(a_{2},a_{4})^p\le d(a_{1},a_{2})^p+d(a_{2},a_{3})^p+d(a_{3},a_{4})^p+d(a_{4},a_{1})^{p}.
\end{equation}
Put
$$r(\mathcal M)=\sup\{p :\mathcal  M\;   \textrm{has roundness $p$}\}.$$

Notice that by a classical tensorization argument, roundness $p$ implies Enflo type $p$ with constant $1$.

\begin{lemma}\label{miau} Let $(\mathcal M_{1}, d_{1})$, $(\mathcal M_{2}, d_{2})$ be two metric spaces.
\begin{enumerate}
\item[(i)] If $\mathcal M_{1}\almost \mathcal M_{2}$, then $r(\mathcal M_{1})\ge r(\mathcal M_{2})$.
\item[(ii)] For $0<s<1$, $\ds r(\mathcal M_{1}^{(s)})=\ds\frac{r(\mathcal M_{1})}{s}.$
\end{enumerate}
\end{lemma}

\begin{proof} We only prove (i) and leave (ii) as an exercise. 
Let  $a_{1},a_{2},a_{3},a_{4}$ in $\mathcal M_{1}$ and pick $\epsilon>0$. There exists $f=f_\epsilon\colon \mathcal M_{1}\to \mathcal M_{2}$ and constants  $A,B>0$ such that $A^{-1}d_{1}(x,y)\le d_{2}(f(x),f(y))\le B d_{1}(x,y)$ for all $x,y\in \mathcal M$, with $AB\le 1+\epsilon$. Assume that $\mathcal M_2$ has roundness $p$, then
\begin{align*}
&d_{1}(a_{1},a_{3})^p+d_{1}(a_{2},a_{4})^p  \le  A^p [d_{2}(f(a_{1}),f(a_{3}))^p+d_{2}(f(a_{2}),f(a_{4}))^p] \\
& \le A^p[d_{2}(f(a_{1}),f(a_{2}))^p+d_{2}(f(a_{2}),f(a_{3}))^p +d_{2}(f(a_{3}),f(a_{4}))^p+d_{2}(f(a_{4}),f(a_{1}))^p]\\
&\le  A^pB^p[d_{1}(a_{1},a_{2})^p+d_{1}(a_{2},a_{3})^p +d_{1}(a_{3},a_{4})^p+d_{1}(a_{4},a_{1})^p]\\
& \le (1+\epsilon)^p[d_{1}(a_{1},a_{2})^p+d_{1}(a_{2},a_{3})^p +d_{1}(a_{3},a_{4})^p+d_{1}(a_{4},a_{1})^p].
\end{align*}
Letting $\epsilon\to 0$ we get the roundness $p$ inequality for $\mathcal M_1$.
\end{proof}

\begin{proposition}
Let $0<s<1$ and $(\mathcal M,d)$ be a metric space with finite roundness. Then
$(\mathcal M,d)\notalmost (\mathcal M,d^s)$.
\end{proposition}

\begin{proof} Let $p=r(\mathcal M)$. We have $1\le p<\infty$, hence if $(\mathcal M,d)\almost (M,d^s)$ it follows from  Lemma~\ref{miau} that $p\ge{p}/{s}$, which contradicts  $0<s<1$.
\end{proof}

Combining with the characterization of metric spaces with infinite roundness from \cite{Westonandall}, gives:

\begin{corollary}
If $(\mathcal M,d)$ is not an ultrametric space then for every $0<s<1$, $(\mathcal M,d)\notalmost (\mathcal M,d^s)$.
\end{corollary}

\section{Embedding snowflakings of the real line}\label{Section3}

\noindent  Here we consider embeddings of power transforms of the Euclidean real line, namely $\redp$ for $0<p\le 1$ where $d_p$ is the $p$-th power of the absolute value. 

First, notice that $\redp$ isometrically embeds into $\mellp$. If $0<p,q\le 1$ the identity map between $\redp$ and $\redq$ is simultaneously a coarse and uniform equivalence and therefore $\redp$ is uniformly and coarsely embeddable into $\mellq$.

 Now, if $0<q<p\le 1$ this can be expressed in terms of snowflakings. Indeed the identity mapping from $\redq$ into $\redp$ satisfies $d_{p}(x,y)=d_{q}(x,y)^{p/q}$, where $p/q\le 1$. It implies that the $p/q$-snowflaked version of $(\R,d_{q})$ embeds isometrically into $\redp$, therefore into $\mellp$.

 If $0<p\neq q\le 1$, it is easy to see that $\redp$ does not admit a Lipschitz copy of $\redq$ using either the Enflo-type or the Hausdorff dimension argument. Indeed the Euclidean real line has (supremum) Enflo-type 2 and Hausdorff dimension 1. Actually we can prove a much  stronger result.

\begin{proposition}\label{nolabel}
If $0<p<q\le 1$, there is  no nonconstant Lipschitz map from $\redq$ into $\mellp$ and, consequently, there is no Lipschitz embedding  from $\redq$ into $\mellp$.
\end{proposition}

\begin{proof}  Let $0<p<q\le 1$ and suppose there is $\varphi: \mathbb R\to \ell_{p}$ that satisfies a Lipschitz condition
$$
\Vert \varphi(s)-\varphi(t)\Vert_{p}^{p}\le K \vert s-t\vert^{q}, \qquad \forall s,t \in \mathbb R,
$$
for some $K>0$. Without loss of generality we assume $\varphi(0)=0$.
 We compose $\varphi$ with $x^{\ast}\in \ell_{p}^{\ast}$ of norm $\Vert x^{\ast}\Vert \le 1$ to obtain
$$
|x^{\ast}\circ\varphi(s)-x^{\ast}\circ\varphi(t)|\le \Vert x^{\ast}\Vert \Vert \varphi(s)-\varphi(t)\Vert_{p} \le K^{1/p} \vert s-t\vert^{q/p}, \qquad \forall s,t\in \mathbb R.
$$
Since $(\mathbb R, |\cdot|)$ is a metrically convex space and  $q/p>1$ we deduce that $x^{\ast}\circ\varphi(t)=0$ for all $t\in [0,1]$. But $\ell_{p}^{\ast}$ separates the points of $\ell_{p}$, which forces $\varphi$ to be $0$. \end{proof}

\begin{remark} The Enflo-type argument is inconclusive in this situation since $\redq$ has Enflo-type $2/q$ and $\mellp$ has Enflo-type $1$.
\end{remark}

For doubling metric spaces (in particular the real line) Assouad  \cite{Assouad1983} proved the following deep result:
\begin{theorem}[Assouad]
Let $0<s\le 1$ and $(\mathcal{M},d)$ be a doubling space. Then $(\mathcal{M},d^s)\lip \ell_2^N=(\R^N,\Vert\cdot\Vert_2)$ for some $N\ge 2$.
\end{theorem}

The proof of Assouad's theorem is subtle and a special attention is given on estimating the dimension of the target space (see \cite{Kahane1981}, \cite{Talagrand1992}, \cite{NaorNeiman}, \cite{AbrahamBartalNeiman2011}  for refinements and discussion of this result). In this paper we deal mainly with infinite-dimensional target spaces and we do not need the full power of Assouad's embedding. The proof of the next theorem makes use of a simplification of the Assouad-type embedding since we allow infinitely many maps, hence infinitely many coordinates.

\begin{theorem}\label{ellpsnow}
For $0<p<q$ there exist real-valued maps $(\psi_{j,k})_{(j,k)\in \Z}$ and positive constants $A_{p,q}, B_{p,q}$ such that 
\begin{equation}\label{longproof} A_{p,q}\vert x-y\vert^p\le \sum_{k\in\Z}\sum_{j\in\Z}\vert \psi_{j,k}(x)-\psi_{j,k}(y)\vert^q\le B_{p,q}\vert x-y\vert^p,
\end{equation}
for all $x,y\in \R.$
\end{theorem}

We will present in Section~\ref{Section4} a nice application of Theorem \ref{ellpsnow} to the embeddability of snowflakings of the spaces $\ell_{p}$, which in certain cases can also be derived from Assouad's theorem (see Remark \ref{remark}). We wrote the proof of Theorem \ref{ellpsnow} with a ``wavelet flavor'' since even if we do not know whether or not there are other connections with approximation theory it seems that the maps obtained or some modification of them might be use to derive interesting embeddings.
\begin{proof}
Let $\psi:\R \to \R_{+}$ be  given by $\ds \psi(t)  = \begin{cases}
                            t+2 &\text{if}\; -2\le t\le 0,\\
                            2-t &\text{if}\; 0\le t\le 2,\\
                            0 & \text{otherwise}.
                            \end{cases}$

\noindent Let $\beta\in\R$ to be chosen later. Define the functions
 $$\psi_{j,k}(t)=2^{k\beta-1}\psi(2^k t-j),\quad \text{for } j,k\in\Z.$$ That is, 

$$
\psi_{j,k}(t)  =\begin{cases}
                            2^{k\beta}2^{k-1}\left(t-\ds\frac{j-2}{2^k}\right)&\text{if}\;\; \ds\frac{j-2}{2^k}\le t<\frac{j}{2^k},\\
                          -2^{k\beta}2^{k-1}\left(t-\ds\frac{j+2}{2^k}\right)&\text{if}\;\; \ds\frac{j}{2^k}\le t<\frac{j+2}{2^k},\\
                            0 & \text{otherwise}.
                          \end{cases}$$
                          Each $\psi_{j,k}$ is supported on the interval $\Big[\frac{j-2}{2^k},\frac{j+2}{2^k}\Big]$, and  fulfills these two estimates:
\begin{equation}\label{FACT1}
\psi_{j,k}(t)\le 2^{k\beta},\, \qquad \forall\, t\in \R,
\end{equation}
and 
\begin{equation}\label{FACT2}
\vert \psi_{j,k}(s)-\psi_{j,k}(t)\vert\le 2^{k-1}2^{k\beta}\vert s-t\vert,\qquad \forall\, s,t\in \R.
\end{equation}
Inequality \eqref{FACT1} follows directly  from the definition of the maps $(\psi_{j,k})_{j,k}$. We remark that $\psi$ is $1$-Lipschitz, and therefore $\psi_{j,k}(t)=2^{k\beta-1} \psi(2^k t -j)$ is $2^{k\beta-1} 2^k $-Lipschitz. This proves \eqref{FACT2}.

\medskip

Now we prove two types of upper-bound estimates using \eqref{FACT1} and \eqref{FACT2}. 

Given $s<t$ pick $K\in\Z$ such that $2^{-(K+1)}\le \vert s-t\vert\le 2^{-K}$.
\medskip

- Upper bound estimate of first type (based on  \eqref{FACT1}):
 \begin{align*}
\vert \psi_{j,k}(s)-\psi_{j,k}(t)\vert^q & \le  (2^{k\beta}+2^{k\beta})^q\\
                                                 & \le   2^q\cdot2^{kq\beta}\\
                                                 & \le   2^q\cdot2^{kq\beta}\cdot2^{(K+1)p}\cdot2^{-(K+1)p}\\
                                                 & \le  2^q\cdot2^{kq\beta}\cdot2^{(K+1)p}\vert s-t\vert^p\\
                                                 & \le  2^{p+q}\cdot2^{kq\beta}\cdot2^{Kp}\vert s-t\vert^p.
 \end{align*}

- Upper bound estimates of second type (based on \eqref{FACT2}):
 \begin{align*}
\vert \psi_{j,k}(s)-\psi_{j,k}(t)\vert^q & \le  2^{q(k-1)}\cdot2^{qk\beta}\vert s-t\vert^q\\
                                                 & \le   2^{q(k-1)}\cdot2^{qk\beta}\vert s-t\vert^{q-p}\vert s-t\vert^p\\
                                                 & \le   2^{q(k-1)}\cdot2^{qk\beta}2^{-K(q-p)}\vert s-t\vert^p\\
                                                 & \le   2^{-q}\cdot2^{kq(1+\beta)}\cdot2^{K(p-q)}\vert s-t\vert^p.
    \end{align*}
Armed with the above inequalities we are ready to substantiate \eqref{longproof}.

\medskip

\noindent {\sc The lower-bound estimate in  \eqref{longproof}}. Note that the support of the function $\psi_{j,K+1}$ is of size exactly $2^{-(K-1)}$. Therefore we can find $j$ such that $s,t\in \supp(\psi_{j,K+1})$. We pick the largest such $j$ that we denote $J$. By our choice of $J$, we force $s$ to belong to $\left[\frac{J-2}{2^{K+1}},\frac{J-1}{2^{K+1}}\right]$ and $t$ to lie in $[\frac{J-1}{2^{K+1}},\frac{J+1}{2^{K+1}}]$. We consider two cases:

\medskip

-  If $\ds t\in\Big[\frac{J-1}{2^{K+1}},\frac{J}{2^{K+1}}\Big]$,
 \begin{align*}
\vert \psi_{j,K+1}(s)-\psi_{j,K+1}(t)\vert^q & = 2^{qK}\cdot2^{q(K+1)\beta}\vert s-t\vert^q\\
                                                  & =  2^{qK}\cdot2^{q(K+1)\beta}\vert s-t\vert^{q-p}\vert s-t\vert^p\\
                                                  & \ge  2^{qK}\cdot2^{q(K+1)\beta}\cdot2^{-(K+1)(q-p)}\vert s-t\vert^p\\
                                                  & \ge  2^{q(\beta-1)+p}\cdot2^{K(q\beta+p)}\vert s-t\vert^p.
                                             \end{align*} 
- If $\ds t\in\Big[\frac{J}{2^{K+1}},\frac{J+1}{2^{K+1}}\Big]$, then $s\notin \supp(\psi_{J+1,K+1})$, hence
  \begin{align*}
\vert \psi_{J+1,K+1}(s)-\psi_{J+1,K+1}(t)\vert^q & =  \vert \psi_{J+1,K+1}(t)\vert^q\\
                                                          & =  2^{qK}\cdot2^{q(K+1)\beta}\left\vert t-\frac{J-1}{2^{K+1}}\right\vert^q\\
                                                          & \ge  2^{qK}\cdot2^{q(K+1)\beta}\cdot2^{-(K+1)q}\\ 
                                                          & \ge  2^{qK}\cdot2^{q(K+1)\beta}\left(\frac{\vert s-t\vert}{2}\right)^q\\
                                                          & \ge  2^{-q}\cdot2^{qK}\cdot2^{q(K+1)\beta}\vert s-t\vert^{q-p}\vert s-t\vert^p\\
                                                          & \ge  2^{-q}\cdot2^{qK}\cdot2^{q(K+1)\beta}\cdot2^{-(K+1)(q-p)}\vert s-t\vert^p\\
                                                          & \ge  2^{q(\beta-2)+p}\cdot2^{K(q\beta+p)}\vert s-t\vert^p.
                                             \end{align*} 
It becomes pretty clear that if we want a Lipschitz lower estimate we are forced to choose $\beta=-\frac{p}{q}$, which we do from now on. 
We remark that $A_{p,q}=2^{-2q}$.

\smallskip
 
\noindent  {\sc The upper-bound estimate in \eqref{longproof}}.
Taking $\beta=-p/q$ the upper estimates of first and second type become

$$\vert \psi_{j,k}(s)-\psi_{j,k}(t)\vert^q \le  2^{p+q}\cdot2^{p(K-k)}\vert s-t\vert^p,$$
and
$$\vert \psi_{j,k}(s)-\psi_{j,k}(t)\vert^q \le 2^{-q}\cdot2^{(q-p)(k-K)}\vert s-t\vert^p.$$                                 
Notice  that for $k$ fixed, $s$ or $t$ belong to the support of $\psi_{j,k}$ for at most 8 values of $j$. All the other contributions in the sum over $j$ are zero. Hence,

\begin{align*}
\sum_{k\in\Z}\sum_{j\in\Z}\vert \psi_{j,k}(s)-\psi_{j,k}(t)\vert^q
&= \sum_{k\le K}\sum_{j\in\Z}\vert \psi_{j,k}(s)-\psi_{j,k}(t)\vert^q+\sum_{k>K}\sum_{j\in\Z}\vert \psi_{j,k}(s)-\psi_{j,k}(t)\vert^q\\
                                                 & \le   8\left( \sum_{k\le K}2^{-q}\cdot2^{(q-p)(k-K)}+ \sum_{k>K}2^{p+q}\cdot2^{p(K-k)}\right)\vert s-t\vert^p\\
                                                 & \le   8\left(2^{-q}\sum_{N\ge 0}2^{-(q-p)N}+2^{p+q}\sum_{N>0}2^{-Np}\right)\vert s-t\vert^p\\
                                                 & \le   8\left(\ds2^{-q}\frac{2^{q-p}}{2^{q-p}-1}+2^{p+q}\ds\frac{2^{p}}{2^{p}-1}\right)\vert s-t\vert^p.
 \end{align*}
We thus get $B_{p,q}=\ds8\left(\frac{1}{2^q-2^p}+\ds\frac{2^{p+q}}{2^p-1}\right)$, and the proof is over.
\end{proof}

\begin{remark} An immediate consequence of Theorem \ref{ellpsnow} is that for $0<p<q\le1$, $\redp\lip\mellq$. The embedding  $x\mapsto (\psi_{\varphi(n)}(x)-\psi_{\varphi(n)}(0))_{n\in\N}$, where  $\varphi$ is any enumeration of $\Z\times\Z$, does the job.

\medskip
\end{remark} 

\section{Lipschitz embeddings between   $L_p$-spaces and $\ell_{p}$-spaces for $0<p<\infty$}\label{Section5}

\noindent A Lipschitz function from a separable Banach space $X$ to a Banach space $Y$ with the Radon-Nikod\'{y}m property, (RNP) for short,  is G\^ateaux differentiable at least at one point. This  important theorem, proved independently by Aronszajn, Christensen, and Mankiewicz \cite{Aronszajn1976, Christensen1973, Mankiewicz1973},  in combination with the simple fact that if a Lipschitz embedding between Banach spaces is differentiable at some point then its derivative at this point is a linear into isomorphism, proves the impossibility of certain Lipschitz embeddings. Thus, for Banach spaces the linear theory (cf.\ \cite{AlbiacKalton2006}) yields:\newline

\noindent (i) If $1\le p,q<\infty$ with $p\not=q$, then $\ell_{p}\notlip \ell_{q}$.

\noindent (ii) Unless $p=q=2$, $L_{q}\notlip \ell_{p}$.

\noindent (iii) If $1\le q<\infty$ then $L_{p}\notlip L_{q}$ unless $1\le q\le p\le 2$ or $p=q$.

\noindent Note that the case of $L_{1}$ as a potential target space of a Lipschitz embedding is special because $L_{1}$ does not have  (RNP). Nevertheless, it still holds that $L_{1}$ does not contain any subset Lipschitz equivalent to $L_{p}$ for $p>2$  because of a cotype obstruction: if $L_{p}\lip L_{1}$ then $L_{p}$ would have been isomorphic to a subspace of $L_{1}^{\ast\ast}$, which has  cotype 2 versus the cotype of $L_{p}$ for $p>2$ that is only $p$.

\medskip

For future reference and the convenience of the reader we summarize  some  facts that follow immediately from the above and that will be used repeatedly from now on.

\begin{proposition}\label{differentiationargument} Let $\mathcal M$ be a metric space.

\begin{enumerate} 
\item[(i)] If  $\mathcal M$ contains a Lipschitz copy of $L_1$ then $\mathcal M$ cannot be Lipschitz embeddable into a Banach space with (RNP).

\item[(ii)] If  $\mathcal M$ contains a Lipschitz copy of a Banach space $X$ and $\mathcal M$ admits a Lipschitz embedding into a Banach space $Y$ with (RNP), then $X$ embeds isomorphically into $Y$.
\end{enumerate}
\end{proposition}

Eventually, throughout this section it  will also be  helpful to be aware of the following recent embedding results that can be found in \cite{Albiac2008}. 

\begin{theorem}\label{embeddings} Let $0<p<q\le 1$. Then: 
\begin{enumerate}
\item[(i)] $\mellp\lip \mellq$.
\item[(ii)] $\mellq\notlip \mellp$.
\item[(iii)] $\mellp\isometric \mclp$.\\
\item[(iv)] $\mclp\isometric (L_1,d_1)=L_1$ and $L_1\isometric\mclp$.
\item[(v)] $\mclp\isometric\mclq$ when $0<p,q\le 1$.
\end{enumerate}
\end{theorem}

Here and subsequently, if $X$ and $Y$ are Banach spaces, $X\isom Y$  will denote the existence of a (linear) isomorphic embedding from $X$ into $Y$, and $X\lisometric Y$ will stand for a  linear isometric embedding. We will write $X\equiv Y$ if there exists a linear onto isometry between them,  and $X\approx Y$ if they are linearly isomorphic.

\subsection{Lipschitz nonembeddability of $L_p$ into $\ell_q$}\

\medskip

\noindent The first question we tackle is the embeddability of $L_p$ into $\ell_q$ for $0<p,q<\infty$. The outcome is crystal clear  since no embedding is possible.

\begin{proposition}
Let $0<p,q<\infty$, $p,q\neq 2$. Then, endowed with their ad-hoc metrics, $L_p\notlip\ell_q.$
\end{proposition}

\begin{proof} In view of the introductory background of the section, there remain  two possible  scenarios to examine.

(a) The mixte regime, i.e., $0<p<1\le q$. From Theorem~\ref{embeddings}, $L_1\isometric\mclp$, hence if $\mclp\lip\ell_q$ then we would have $L_1\lip\ell_q$, which is in contradiction with Proposition~\ref{differentiationargument} (i).  If $0<q<1\le p$ it suffices to know that $\mellq\lip\ell_1$, hence if $L_p\lip\mellq$ then $L_p\lip\ell_1$, a contradiction.
\smallskip

(b) The $\mathsf F$-space regime,  i.e., $0<p,q<1$. Since $L_1\isometric\mclp$, if $\mclp\lip\mellq$ then $L_1\lip\ell_1$, which is impossible as we know from linear theory.
\end{proof}

\subsection{Lipschitz embeddability of $\ell_p$ into $\ell_q$}\

\medskip

\noindent The following proposition tells us that when trying to Lipschitz embed $\ell_p$ into $\ell_q$,  this is only possible in the $\mathsf F$-space regime under some restriction on the values of $p$ and $q$.
\begin{proposition}\label{Proplpintolq}\ 
\begin{enumerate}
\item[(i)] If $1\le p,q<\infty$, then $\ell_p\notlip\ell_q$.\\
\item[(ii)] If $0<p<1<q$, then $\ell_q\notlip\mellp$ and $\mellp\notlip\ell_q$.\\
\item[(iii)] If $0<p<q\le 1$, then $\mellp\lip\mellq$ but $\mellq\notlip\mellp$.
\end{enumerate}
\end{proposition}

\begin{proof} Only   (ii)  requires a proof.  Note that $\mellp\lip\ell_1$, hence if $\ell_q\lip\mellp$ then we would have $\ell_q\lip\ell_1$, a contradiction. For the other statement,
 if $\mellp\lip\ell_q$ then passing to ultraproducts we have that $(\mellp)_{\mathcal{U}}\lip L_q(\nu)$.
But it follows from Naor's master thesis \cite{Naorthesis} that $(\mellp)_{\mathcal{U}}$ contains a Lipschitz copy of $L_1$, therefore so does  $L_q(\nu)$. Contradiction since $q>1$.
\end{proof}

\subsection{Lipschitz embeddability of $\ell_p$\ into $L_q$}\label{Section4.3}\

\medskip

\noindent This is the state of affaires. 

\begin{proposition}\label{FlorentProposition5.4}\
\begin{enumerate}
\item[(i)] Suppose $1\le p,q<\infty$. 
\begin{itemize}
                                              \item[(a)] If $2<p\neq q$, then $\ell_p\notlip L_q$;
                                              \item[(b)] If $1\le p<q\le 2$, then $\ell_q\lisometric L_p$ but $\ell_p\notlip L_q$;
                                              \item[(c)] If $1\le p<2<q$, then $\ell_p\notlip L_q$ and $\ell_q\notlip L_p$.
\end{itemize}                                             
\item[(ii)] Suppose $0<p<1\le q$. Then, $\mellp\notlip L_q$ and \begin{itemize}
                                          \item[(a)] if $0<p<1\le q\le 2$, $\ell_q\isometric\mclp$;
                                          \item[(b)] if $0<p<1<2<q$, $\ell_q\notlip\mclp$.
                                          \end{itemize}
\item[(iii)] If $0<p,q\le 1$, then $\mellp\isometric\mclq$.
\end{enumerate}
\end{proposition}

\begin{proof} (i) follows from the reminders we made at the beginning of this section between the Lipschitz and the linear structure of $L_{p}$ for $p\ge 1$.

 For $0<p<1$, the fact that $\mellp\notlip L_q$ when $1\le q$ can be proved using an ultraproduct argument like in the proof of Proposition~\ref{Proplpintolq} (ii). To see (ii) (a), it suffices to recall that $\ell_q\lisometric L_q \lisometric L_1$ for the range $1<q\le 2$ and use the embedding $L_1\isometric\mclp$ of Theorem~\ref{embeddings} (iv). For (ii) (b) notice that $\mclp\isometric L_1\coarse\ell_2$, hence if $\ell_q\lip\mclp$ then $\ell_q\coarse\ell_2$. But this is impossible using a metric cotype argument \cite{MendelNaor2008} or the result of Johnson and Randrianarivony that appeared in \cite{JohnsonRandrianarivony2006} .

 Part (iii)  follows from the  diagram $\mellp\isometric\mclp\isometric\mclq.$
\end{proof}

\subsection{Lipschitz embeddability of  $L_p$ into $L_q$}\

\medskip

\noindent The situation is the exact same one as in \textsection\ref{Section4.3}. The proof  of the following proposition goes along the same lines as  the proof of the Proposition~\ref{FlorentProposition5.4}, and so we omit the details.

\begin{proposition}\label{usedtobeProp5.5}\ \\
\begin{enumerate}
\item[(i)] Suppose $1\le p,q<\infty$. Then, \begin{itemize}
                                              \item[(a)] if $2<p\neq q$, $L_p\notlip L_q$;
                                              \item[(b)] if $1\le p<q\le 2$, $L_q\lisometric L_p$ but $L_p\notlip L_q$;
                                              \item[(c)] if $1\le p<2<q$, $L_q\notlip L_p$ and $L_p\notlip L_q$.
                                              \end{itemize}
\item[(ii)] Suppose $0<p<1<q$. Then, $\mclp\notlip L_q$ and \begin{itemize}
                                          \item if $0<p<1<q\le 2$, $L_q\isometric\mclp$;
                                          \item if $0<p<1<2<q$, $L_q\notlip\mclp$.
                                          \end{itemize}
\item[(iii)] If $0<p,q\le 1$, then $\mclp\isometric\mclq$.
\end{enumerate}
\end{proposition}

\subsection{Application to unique Lipschitz subspace structure problems.}\

\medskip

\noindent Let $X$  and $Y$ be  $\mathsf F$-spaces. The space $X$ is said to have a {\it  unique Lipschitz $\mathsf F$-subspace structure} if  the following equivalence holds:
\begin{equation}\label{LipsubsProb}Y\lip X \Longleftrightarrow Y\isom X.\end{equation}

\smallskip

If we only allow $Y$ to be a Banach space we will refer to this problem as the {\it unique Lipschitz subspace structure problem for $\mathsf F$-spaces}. 

The {\it unique Lipschitz subspace structure problem for Banach spaces} is classical and has been thoroughly investigated. We refer to \cite{BenyaminiLindenstrauss2000} for a detailed account.

A linear isomorphic embedding between Banach spaces is automatically a Lipschitz embedding, hence, in that case, one implication in \eqref{LipsubsProb} is trivial.
This is no longer true for $\mathsf F$-spaces and both implications have to be independently checked. 

In the category of Banach spaces the unique Lipschitz subspace structure problem is still widely open. In fact, there are separable Banach spaces (such as $c_{0}$) with a unique Lipschitz structure that fail to have a unique Lipschitz subspace structure. As already mentioned in the introduction of Section~\ref{Section5}, the Radon-Nikod\'{y}m property was clearly identified as being a sufficient condition for a separable space to have a unique Lipschitz subspace structure. It is doubtful that a similar strategy could be carried out in the $\mathsf F$-space framework.

We will now place the results of this section in the context just described and we will analyze and compare the various and relevant versions of the unique Lipschitz subspace structure for the spaces $L_p$ or $\ell_p$ for the entire range $0<p<\infty$. We first discuss the Lipschitz $\mathsf F$-subspace structure problem for the classical function and sequence Banach spaces.

\medskip

\noindent $\bullet$
The reflexive spaces $L_p$ and $\ell_p$ for $p>1$ have  a unique Lipschitz subspace structure.
The theory is consistent if we consider the classical $\mathsf F$-subspaces $L_q$ and $\ell_q$ when $0<q<1$. Indeed, it follows from the results of this section (respectively, \cite{KaltonPeckRoberts1984}) that none of the spaces $L_p$ or $\ell_p$ for  $p>1$ contains a Lipschitz (respectively, isomorphic) copy of $L_q$ nor $\ell_q$ when $0<q<1$.

\medskip

\noindent $\bullet$
As a separable dual, the space $\ell_1$ has a unique Lipschitz subspace structure.
However, the fact that $\mellq\lip \ell_1$ but $\ell_q\notisom \ell_1$  for all $q<1$ shows that $\ell_1$ does not have a unique Lipschitz $\mathsf F$-subspace structure.

\medskip

\noindent $\bullet$
The unique Lipschitz subspace structure problem for $L_1$ is still open. As it happens, the space $L_1$ does not have a unique Lipschitz $\mathsf F$-subspace structure since $\mclq\lip L_1$ for all $0<q<1$ but $L_q\notisom L_1$. 

\medskip

The case of $\ell_{1}$ shows that there can be Banach spaces with (RNP) that do not have a unique Lipschitz $\mathsf F$-subspace structure, but so far  the following question is still unanswered.

\begin{question}
Does a reflexive Banach space have a unique Lipschitz $\mathsf F$-subspace structure? 
\end{question}

Let us turn to the Lipschitz subspace structure problem for the classical function and sequence $\mathsf F$-spaces.

\medskip

\noindent $\bullet$ The situation for the spaces $\ell_q$ when $q\in(0,1)$ is simple. They do not contain neither a Lipschitz nor a  isomorphic copy of any Banach space. This follows from \cite[Corollary 2.8]{KaltonPeckRoberts1984} for the isomorphic embedding and Proposition~\ref{nolabel} for the Lipschitz one. Thus they have trivially a unique Lipschitz subspace structure.

\medskip

\noindent $\bullet$ The case of $L_q$ for $q\in(0,1)$ is more subtle. Recall that every Banach space has cotype $\ge 2$. If $Y$ is a Banach space with cotype {\it strictly } greater than $2$, then there is neither a Lipschitz nor a isomorphic embedding of $Y$ into $L_q$. Indeed, assume  there is an isomorphic embedding $T:Y\to L_{q}$. Since the topologies induced in $L_{q}$ by the distance and the quasi-norm  are uniformly equivalent then $T$ is a linear map satisfying an inequality of the form $$\Vert x-y\Vert_Y\lesssim \Vert Tx-Ty\Vert_q\lesssim \Vert x-y\Vert_
Y.$$
After raising $q$th-power we obtain that $(Y,\Vert\cdot\Vert_Y^q)\lip \mclq$. This gives a coarse embedding of $(Y,\Vert\cdot\Vert_Y)$ into $\mclq$.  But, as already seen in this section, $\mclq$ embeds isometrically into $L_1$, which in turn coarsely embeds into a Hilbert space. We thus reach a contradiction since a Banach space with cotype {\it strictly } greater than $2$ cannot coarsely embed into a Hilbert space. We can skip the first step in the proof to get the same conclusion if $T$ is just a Lipschitz embedding.

Now if $Y$ is one of the spaces $L_p$ for some $p\in[1,2]$, it has cotype $2$. Using $p$-stable random variables we can construct an explicit isomorphic embedding of $Y$ into $L_q$. Composing the isometric embedding of $Y$ into $L_1$ (provided again by $p$-stable random variables) with the isometric embedding of $L_1$ into $\mclq$ of the present section we get an isometric embedding into $\mclq$. Unfortunately, we do not know how to handle the other Banach spaces $Y$ that have cotype $2$ (e.g., the dual of the James space). This suggests:

\begin{question}
Do the classical $\mathsf F$-spaces $L_q$ for $0<q<1$ have a unique Lipschitz subspace structure?
\end{question} 

Finally, let us mention that none of the classical $\mathsf F$-spaces $L_q$ nor $\ell_q$ when $0<q<1$ has a unique Lipschitz $\mathsf F$-subspace structure. Indeed, each of them contains a Lipschitz copy of an $\mathsf F$-space of the same family that cannot be isomorphic to a linear subspace.

\section{Embedding snowflakings of $L_p$ and $\ell_{p}$ for $0<p<\infty$}\label{Section4}

\noindent  We have been able to track down essentially two results of this kind in the literature. The first one is due to Bretagnolle and al., who in  \cite{Bretagnolleetal1965}  proved that $(L_p,\Vert\cdot\Vert^{p/q})$ is isometric to a subset of $L_q$ for $1\le p<q\le 2$. Later,  Mendel and Naor \cite{MendelNaor2004} generalized this result. Indeed they observed that, since the complex space $L_{q}$ embeds isometrically as a real space into the real space $L_{q}$, in order to  embed $L_{p}$ in $L_{q}$ for $p<q$ it suffices to embed $L_{p}(\mathbb R)$ into the complex space $L_{q}(\mathbb R\times \mathbb R;\mathbb C)$. This is accomplished via the map
\[
T: L_p(\R) \to L_q(\R\times\R;\mathbb C),\qquad
f  \mapsto  T(f)(s,t)=\ds c\frac{1-e^{itf(s)}}{\vert t\vert^{(p+1)/q}},
\]
where
$$
c^{-q}=2^{q/2}\left(\int_{-\infty}^{\infty}\frac{(1-\cos(u))^{q/2}}{\vert u\vert^{p+1}}\,dt\right).
$$
Therefore for $L_{p}$-spaces we have:

\begin{theorem}[Mendel and Naor \cite{MendelNaor2004}]\label{MendelNaooor}\ 
\begin{enumerate}
\item[(i)] If $0< p<q\le 1$, then  $\mclp\isometric \mclq$.
\item[(ii)] If $0< p<1<q$, then $(L_p,d_p^{1/q})\isometric L_q$.
\item[(iii)] If $1\le p\le q$, then $(L_p,\Vert\cdot\Vert_p^{p/q})\isometric L_q$.
\end{enumerate}
\end{theorem}

\subsection{Embedding snowflakings of $\ell_p$-spaces}\

\medskip

\noindent The aim of this section is to utilize techniques from Section~\ref{Section3} to give a simple explicit Lipschitz embedding between the $\ell_{p}$-spaces and some of their snowflakings. Unfortunately, unlike the case of $L_p$-spaces we do not know if there is an isometric version of the following proposition:

\begin{proposition}\label{sellp}\ 
\begin{enumerate}
\item[(i)] If $0<p<q\le1$, then $\mellp\lip\mellq$.
\item[(ii)] If $0< p\le1<q$, then $(\ell_p,d_p^{1/q})\lip \ell_q$.
\item[(iii)] If $1\le p\le q$, then $(\ell_p,\Vert\cdot\Vert_p^{p/q})\lip \ell_q$.
\end{enumerate}
\end{proposition}

\begin{proof}
Let  $$\ell_p(\N\times\Z\times\Z,\R)=\Big\{ (z_{i,j,k})_{(i,j,k)\in\N\times\Z\times\Z}\in \R^{\N\times\Z\times\Z}:\ \sum_{i\in\N}\sum_{j\in\Z}\sum_{k\in\Z}\vert z_{i,j,k}\vert^p<\infty\Big\},$$
and define the mapping
\[
f:\ell_p(\N,\R) \to \ell_q(\N\times\Z\times\Z,\R),\quad
(x_i)_{i\in\N} \mapsto  (\psi_{j,k}(x_i)-\psi_{j,k}(0))_{(i,j,k)\in\N\times\Z\times\Z}.
\]
Applying  Theorem \ref{ellpsnow} coordinate-wise yields that for all $(x_i),(y_i)\in \ell_p(\N,\R)$, 
\begin{equation}\label{eqpag11.1}
A_{p,q}\sum_{i\in\N}\vert x_i-y_i\vert^p\le \sum_{i\in\N}\sum_{j\in\Z}\sum_{k\in\Z}\vert \psi_{j,k}(x_i)-\psi_{j,k}(y_i)\vert^q\le B_{p,q}\sum_{i\in\N}\vert x_i-y_i\vert^p.\end{equation}
Then:

\noindent$\circ$  If $0<p<q\le 1$,  inequality \eqref{eqpag11.1} tells us exactly that $$\mellp\lip(\ell_q(\N\times\Z\times\Z,\R),d_q)\equiv\mellq.$$
\noindent $\circ$ If $0<p<1<q$, raising to the power $1/q$ we get $$(\ell_p,d_p^{1/q})\lip\ell_q(\N\times\Z\times\Z,\R)\equiv\ell_q.$$
\noindent $\circ$  If $1\le p<q$, raising to the power $1/q$  and writing $1/q=\frac{1}{p}\cdot \frac{p}{q}$ we obtain $$(\ell_p,\Vert\cdot\Vert_p^{p/q})\lip\ell_q(\N\times\Z\times\Z,\R)\equiv\ell_q.$$
\end{proof}

\begin{remark}\label{remark}\ \\
\noindent $\bullet$ One can apply Assouad's theorem almost as a black box to prove Proposition \ref{sellp} in the case where $(0<p<q)$ and $(q\ge 1)$ as follows: the real line $(\R, \vert\cdot\vert)$ is a 
doubling space, and therefore for every $0<s<1$ and in particular $s=p/q$, $(\R, \vert\cdot\vert^s)$ embeds bi-Lipschitzly for any $1\le q<\infty$ into $\ell_q^N$ for some dimension $N$ by an appeal to the equivalence of finite-dimensional  norms. Then taking $\ell_q$-sums, where the $\ell_q$-sum of a sequence of metric linear spaces is defined in the obvious way, we get the desired embeddings.

\medskip

\noindent $\bullet$ The distortion of the embeddings in Proposition \ref{sellp} blows up when $q$ is close to $p$ and we do not know if we can find embeddings without this, a priori, unexpected behavior. 

\medskip

\noindent$\bullet$ For the sake of completeness we include another well known isometric embedding of the $1/2$-snowflaked version of $L_1$ into a Hilbert space.
\[
T:L_1(\R,\R) \to L_2(\R\times\R,\R),\quad
f  \mapsto  T(f)(s,t)= \begin{cases}
                                      1 &\textrm{ if }\; 0\le t\le\ f(s),\\
                                     -1&\textrm{ if }\; f(s)<t<0,\\
                                                                           0&\textrm{ otherwise}.
                                      \end{cases}
\]
\end{remark}

\subsection{On the membership of  $L_p$ and $\ell_{p}$  in the classes $\sd$ and $\nst$}\label{Section4.2}\

\medskip

\noindent It follows easily from Aharoni \cite{Aharoni1974} that any snowflaked version of $c_0$ Lipschitz embeds  into $c_0$ itself.
Similarly, it can be easily derived from the work of Schoenberg \cite{Schoenberg1938} that any snowflaking of a Hilbert space isometrically embeds into a Hilbert space. In other words, $c_0$ and $\ell_2$ are in $\sd$. In the next proposition we show that certain $L_p(\mu)$-spaces belong to $\sd$ as well.

\begin{proposition}\ 
\begin{enumerate}
\item[(i)] Let $1\le p< 2$ and ${p}/{2}\le s< 1$. Then $(L_p,\Vert\cdot\Vert^{s})\isometric L_p$.
\item[(ii)] $(L_1,\Vert\cdot\Vert^{s})\isometric L_1$, for all $0<s<1$.
\item[(iii)] For $0<p<1$, $(L_p,d_p^{s})\isometric\mclp$, for all $0<s<1$.
\end{enumerate}
\end{proposition}

\begin{proof} (i)  Fix $1\le p< 2$ and let $1\le p< q\le 2$ and $s={p}/{q}$. Then,  $$(L_p,\Vert\cdot\Vert^{s})\isometric L_q\isometric L_p.$$
 
 We proved (ii) already for ${1}/{2}\le s<1$ in (i). For $0< s<{1}/{2}$ we proceed as follows. Let $0<\lambda<1$. Then $N(x,y)=\Vert x-y\Vert_1^\lambda$ is a negative definite kernel. Hence, by Schoenberg's theorem there is $T\colon L_1\to L_2$ such that $\Vert Tx-Ty\Vert_2^2=\Vert x-y\Vert_1^\lambda$ for all $x,y\in L_{1}$. Therefore $(L_1,\Vert\cdot\Vert^{\frac{\lambda}{2}})\isometric L_2\isometric L_1$, for all $0<\lambda<1$.
 
(iii) Since $\mclp\isometric\ L_1$ we have $(L_p,d_p^s)\isometric L_1^{(s)}\isometric L_1\isometric \mclp$.
\end{proof}

Now we will  investigate which spaces  $L_{p}$ and $\ell_{p}$ belong to the class $\nst$ defined in Section 2. In other words, we want to know if a given space  $L_{p}$ or $\ell_{p}$ can be Lipschitz embedded into one of its snowflakings. It turns out that this question has a negative answer for the entire range $0<p<\infty$. Thanks to the deep results of Naor and Schechtman \cite{NaorSchechtman2002}, UMD Banach spaces are in $\nst$.  Recall that a Banach space $X$ is called a {\it UMD $p$-space} ($1<p<\infty$) if there exists a constant $\gamma_{p,X}$ such that for every finite $L_{p}$-martingale difference sequence $(d_{j})_{j=1}^{n}$ with values in $X$ and every $\{-1,1\}$-valued sequence $(\varepsilon_{j})_{j=1}^{n}$ we have
\[
\left(\mathbb E\Big\Vert\sum_{j=1}^{n}\varepsilon_{j}d_{j}\Big\Vert^{p} \right)^{1/p}\le \gamma_{p,X}\left(  \mathbb E\Big\Vert\sum_{j=1}^{n}d_{j}\Big\Vert^{p}\right)^{1/p}.
\]
It can be shown using Burkholder's good $\lambda$-inequalities that if $X$ is a UMD $p$-space for some $1<p<\infty$, then it is a UMD $p$-space for all $1<p<\infty$, hence a space with this property will simply be called UMD space. Basic examples of UMD spaces are all Hilbert spaces and the spaces $L_{p}(\mu)$ for $1<p<\infty$ where $\mu$ is a $\sigma$-finite measure.   Amongst other things Naor and Schechtman proved that for UMD Banach spaces the notions of Enflo-type and of Rademacher-type coincide. Using this powerful result we can state the following theorem.

\begin{theorem}\label{snowtarget}
 Let $0<p,q<\infty$ and $0<\alpha,\beta\le1$. Let us consider the  function $$ \tau(p) =\begin{cases}
                            1 & \text{if}\;\; 0< p\le 1,\\
                            p  & \text{if}\;\;1\le p\le 2,\\
                            2  & \text{if}\;\; p\ge 2.
                            \end{cases}$$
 If $  {\tau(p)}/{\tau(q)} >{\alpha}/{\beta},$ then
  $\ell_p^{(\alpha)}\notlip L_q^{(\beta)}$.
In particular, the spaces $L_{p}$ and $\ell_{p}$ belong to the class  $\nst$ for all $0<p<\infty$.
\end{theorem}

\begin{proof}
The Enflo-type of $L_p(\mu)$-spaces is given by the function $\tau$ and we apply Proposition \ref{restriction}.
\end{proof}

We next include an alternative way to show  that every space $\ell_{p}$ and $L_p$ belongs to  $\nst$. This proof does not require an Enflo-type argument but relies instead on Mendel-Naor's embedding, our analogue for $\ell_p$-spaces, and the Lipschitz structure of the spaces $\ell_{p}$ and $L_p$ that is  described in Section~\ref{Section5}.

\begin{proposition}\label{snowlptarget} Let $0<s<1$. Then,
\begin{enumerate}
\item[(i)]   $L_p\notlip L_p^{(s)}$ for any $0<p<\infty$.
\item[(ii)] $\ell_p\notlip \ell_p^{(s)}$ for any $0<p<\infty$.
\end{enumerate}
\end{proposition}

\begin{proof} (i) Consider first the case $1\le p<\infty$. Given $0<s<1$, we can write $s={p}/{q}$ for some $p<q<\infty$. Using Theorem~\ref{MendelNaooor} (iii),  $L_p^{({p}/{q})}\isometric L_q$, hence if $L_p\lip L_p^{(s)}$ it would follow that $L_p\lip L_q$, an absurdity. 

Suppose now that $0<p<1$ and let $0<s<1$. We can write $s=1/q$ for some $1<q<\infty$. Another appeal to Theorem~\ref{MendelNaooor} (ii) gives $(L_p,d_p^{1/q})\isometric L_q$, therefore if $(L_p,d_p)\lip L_p^{(s)}$ then we would obtain $(L_p,d_p)\lip L_q$, thus contradicting Proposition~\ref{usedtobeProp5.5} (ii).

To show (ii), we argue exactly as in (i), replacing Theorem~\ref{MendelNaooor} with Proposition~\ref{sellp} in combination with the linear structure of the $\ell_p$-spaces.
\end{proof}

It was pointed out in \cite{Albiac2008}  that the quasi-Banach space $\ell_p$ (respectively, $L_p$) does not Lipschitz embed into the quasi-Banach space $\ell_q$ (respectively, $L_q$) for $0< p<q\le1$. Moreover, the author was able to prove a stronger result for $L_p$-spaces, namely, every Lipschitz map from the quasi-Banach space $\ell_p$ (respectively, $L_p$) into the quasi-Banach space $\ell_q$ (respectively, $L_q$) is constant. The re-statement of these results  in the metric context has the following form. 

\begin{theorem}\label{Albiac} Suppose $0< p<q\le1$. Then,
\begin{enumerate}
\item[(i)] $\mellp\notlip (\ell_q, d_q^{p/q})$.
\item[(ii)] $\mclp\notlip (L_q, d_q^{p/q})$.
\end{enumerate}
\end{theorem}

 Notice that Theorem \ref{snowtarget} extends Theorem~\ref{Albiac}. We end  this section with an alternative proof of the following strengthening of the first assertion in Theorem~\ref{Albiac}.

\begin{proposition}\label{snowfrechettarget} Suppose
$0< p,q\le1$. Then $\ell_{p}\notlip \ell_q^{(s)}$ for any $0<s<1$.
\end{proposition}

\begin{proof}  If $0< q\le p\le 1$, there is actually no Lipschitz map.
  If $0< p\le q\le1$, suppose there is $f:\ell_p\to \ell_q$ so that for some constant $K>0$, 
  $$K^{-1}\Vert x-y\Vert_p^{p/s}\le \Vert f(x)-f(y)\Vert_q^{q}\le K\Vert x-y\Vert_p^{p/s},\quad\forall x,y\in \ell_{p}.$$
Given any $\{x_i\}_{i=1}^n\subset\ell_p$, denote $z_k=x_1+\dots+x_k$ and put $z_0=0$. Then,
\begin{align*}
\Big\Vert\sum_{i=1}^n x_i\Big\Vert_p^{p/s}&=\Vert z_n-z_0\Vert_p^{p/s}\\
  &\le K^s\Vert f(z_n)-f(z_0)\Vert_q^q\\
& \le  K^s\sum_{k=1}^n\Vert f(z_k)-f(z_{k-1})\Vert_q^q\\
& \le  K^{2s}\sum_{k=1}^n\Vert z_k-z_{k-1}\Vert_p^ {p/s}\\
&=K^{2s}\sum_{k=1}^n\Vert x_k\Vert_p^{p/s}.
\end{align*}
By the Aoki-Rolewicz theorem this would imply that space $\ell_{p}$ seen as a quasi-Banach space can  be equipped with an equivalent ${p}/{s}$-quasi-norm,  which is impossible because $p/s>p$ (see \cite[Lemma 2.7]{KaltonPeckRoberts1984}).
\end{proof}
\newpage{}
\section{Concluding remarks and open questions}
The following tables, to be read clockwise summarize the Lipschitz embeddability status of the classical Lebesgue spaces endowed with their ad-hoc metrics or their snowflakings.  
\begin{table}[ht]
\begin{center}
\begin{tabular}{|c|c|c|}\hline
& & \\
$\lip$ & $\ell_q$ &   $L_q$\\
& & \\
 \hline
 & & \\
$\ell_p$         & if and only if $\left\{\begin{array}{l}
                          0<p<q\le 1\\
                          p=q\\
                    \end{array}\right.$  & if and only if  $\left\{\begin{array}{l}
                          0<p<q\le 1\\
                          0<q\le p\le 2\\
                          p=q\\
                          p=2, q>0\\
                    \end{array}\right.$ \\
                    & & \\
& & (isometric) \\ 
& & \\    
 \hline
 & & \\
$L_p$  & Never unless $p=q=2$ & if and only if  $\left\{\begin{array}{l}
                          0<p,q\le 1\\
                          0<q\le p\le 2\\
                          p=q\\
                          p=2, q>0\\
                    \end{array}\right.$ \\ 
                 & & \\   
             &  (isometric) &  (isometric)  \\
             & & \\
 \hline
             & & \\
$\ell_p^{(s)}$& if  $\left\{\begin{array}{l}
                          1\le p\le q \textrm{ and }s=p/q\\
                          0< p\le 1\le q \textrm{ and }s=1/q\\
                          p=q=2 \textrm{ and }0<s\le1\\
                    \end{array}\right.$ 
             &  if  $\left\{\begin{array}{l}
                          1\le p\le q \textrm{ and }s=p/q\\
                          0< p\le 1\le q \textrm{ and }s=1/q\\
                          p=q=2 \textrm{ and }0<s\le1\\
                    \end{array}\right.$\\      
             &  &  \\ 
 \hline
 & & \\
$\isometric$ & $\ell_q$ &   $L_q$\\
& & \\
 \hline
 & & \\
 $L_p^{(s)}$& if $p=q=2 \textrm{ and }0<s\le1$ 
             &  if  $\left\{\begin{array}{l}
                          1\le p\le q \textrm{ and }s=p/q\\
                          0< p\le 1\le q \textrm{ and }s=1/q\\
                          0<p=q\le 1 \textrm{ and }0<s\le1\\
                          p=q=2 \textrm{ and }0<s\le1\\
                          1\le p=q<2 \textrm{ and }p/2\le s\le1\\
                          0< q\le 1\le p\le 2 \textrm{ and }s=q\\

                    \end{array}\right.$\\
 & & \\      
 \hline
 \end{tabular}
 \end{center}
\end{table}

\newpage{}

\begin{table}[ht]
\begin{center}
\begin{tabular}{|c|c|c|}
\hline
$\notlip$ &  $\ell_q^{(t)}$ & $L_q^{(t)}$\\
 \hline
 & & \\
 $\ell_p^{(s)}$&  $\begin{array}{l}
                          \textrm{ if } 0<p,q\le 1 \textrm{ with }\\
                           \\
                          s=1 \textrm{ and } 0<t<1\\
                           \end{array}$ &
             if  $\left\{\begin{array}{l}
                          0<p,q\le 1 \textrm{ and }s<t\\
                          p,q\ge 1 \textrm{ and }\ds\frac{s}{t}<\ds\frac{\min\{p,2\}}{\min\{q,2\}}\\
                          0<p\le 1\le q \textrm{ and }\ds\frac{s}{t}<\ds\frac{1}{\min\{q,2\}}\\
                          0<q\le 1\le p \textrm{ and }\ds\frac{s}{t}<\min\{p,2\}\\
                    \end{array}\right.$\\
 & & \\      
\hline
\end{tabular}
\end{center}
\end{table}

\noindent Our work leaves a myriad of  open questions and sets what we think is a different and alternative approach to the classical study of the nonlinear geometry of Banach spaces. We will introduce several parameters and highlight several questions that we found particularly interesting.

$\bullet$ If one wants to measure quantitatively how close can we possibly be to a Lipschitz embedding between two different $\ell_p$-spaces it is natural to define the following parameter for $0<p\neq q<\infty$ $$s_{p\to q}=\sup\{s\le 1: (\ell_p, d_p^s)\lip (\ell_q, d_q)\},$$

\noindent with the convention that $s_{p\to q}=0$ if the latter set is empty. In this paper we obtained tight estimates for $s_{p\to q}$ in several situations. However, when $1\le p\le 2< q$ we proved that $\frac{p}{q}\le s_{p\to q}\le\frac{p}{2}$ but we do not know if we can close this gap. Similarly, for $2\le p< q$ we have $\frac{p}{q}\le s_{p\to q}\le 1$ and $1$ cannot be attained.
Rather frustrating is the fact that we have no indication on the parameter $s_{p\to q}$ when $0<q<p\le 1$. A weaker question would be:

\begin{question}\label{q1}
If $0< q<p\le 2$, do we have $(\ell_p, d_p^{s})\lip\mellq$ for some $0<s<1$?
\end{question}

$\bullet$ Another natural parameter attached to a metric space $(\mathcal M,d)$ is  $$\sigma_{\mathcal M}=\sup\{t\ge 0:(\mathcal M, d^{1-t})\lip (\mathcal M, d)\}.$$
$\sigma_{\mathcal M}$ is related to the classes $\sd$ and $\nsd$ introduced in Section~\ref{Section2}. For instance, $\sd$ is the class of metric spaces such that $\sigma_{\mathcal M}=1$.
We obtained that $\frac{2-p}{2}\le\sigma_{L_p}\le 1$ if $1\le p<2$ and that $\sigma_{L_p}=1$ for $0<p\le 1$ but we have no estimate whatsoever neither for $\sigma_{L_p}$ when $p>2$ nor for $\sigma_{\ell_p}$ when $0<p<\infty$. A related question is:

\begin{question}\label{q2}
If $0< p<q<1$, do we have $(\ell_p, d_p^{s})\lip\mellq$ for some $0<s<1$?
\end{question}

 We can answer Question \ref{q2} affirmatively if $\sigma_{\ell_p}$ is nontrivial, i.e., $\sigma_{\ell_p}>0$.

\smallskip

$\bullet$ We now introduce a last parameter $$\beta_{\mathcal M}=\sup\{t\ge 0:(\mathcal M, d)\lip (\mathcal M, d^{1-t})\}.$$
This parameter is related to the class $\nst$ and its complement, the class 
$$\complement{\nst}=\Big\{(\mathcal M,d): (\mathcal M,d)\lip (\mathcal M,d^s)\;\text{for some}\,0<s\le1 \Big\},$$ formed by the collection of metric spaces $\mathcal M$ that Lipschitz embed in their own snowflaked versions $\mathcal M^{(s)}$ for some $0<s<1$. Indeed, a metric space belongs to $\nst$ if and only if $\beta_{\mathcal M}=0$.

\smallskip

It is quite clear from the sections above that the class $\nst$ is very large but what can be said about its complement. Possible candidates have to be amongst the metric spaces with infinite Enflo-type for instance. Unfortunately it can be rather complicated to compute or even estimate the Enflo-type of a given metric space. However, we know one family in $\complement{\nst}$, namely, ultrametric spaces.

\begin{question}\label{q3}
Describe the class $\complement{\nst}$? Are there metric spaces $\mathcal M$ other than ultrametrics such that $\beta_{\mathcal M}>0$?
\end{question}

$\bullet$ As a byproduct of our work we obtain a new and more direct path to prove that $\ell_q\coarse\ell_p$ for $1\le p<q\le 2$ by composing  the embedding of the $2/q$-snowflaked version of $\ell_q$ into $\ell_2$ and the coarse embedding from $\ell_2$ into $\ell_p$ of Nowak \cite{Nowak2006} based on Dadarlat-Guentner criterion from \cite{DadarlatGuentner2003}. As we already mentioned this is not possible if $2\le p<q$. On the other hand, we know that $L_p\coarse L_q$ if $0<p< q<\infty$ but the following one remains elusive:
\begin{question}\label{q4}
If $2< p< q<\infty$, does $L_p\coarse\ell_q$?
\end{question}

The answer to this question would be affirmative if we could solve positively the following stronger question:

\begin{question}\label{q5}
If $2< p<\infty$, does $L_p\coarse\ell_p$?
\end{question}

\medskip
\noindent{\it Aknowledgements.} The second author is very grateful to Bill Johnson and Gideon Schechtman for various enlightening discussions on the topics covered in this article.

\end{document}